\documentclass[12pt]{amsart}

\usepackage{amsfonts}
\usepackage{calligra,mathrsfs}
\usepackage{color}
\usepackage{amscd}
\usepackage{epsfig}
\usepackage{graphicx}
\usepackage{amsmath,amssymb}
\usepackage{amsthm}
\usepackage{url}
\usepackage[normalem]{ulem}

\newtheorem{thm}{Theorem}[section]
\newtheorem{lem}[thm]{Lemma}
\newtheorem{hypothesis}[thm]{Hypothesis}

\newtheorem{prop}[thm]{Proposition}
\newtheorem{cor}[thm]{Corollary}
\newtheorem{defn}[thm]{Definition}

\newtheorem{rmk}[thm]{Remark}
\newtheorem{rmks}[thm]{Remarks}
\newtheorem{ques}[thm]{Question}

\newtheorem{ex}[thm]{Example}

\def\cd2{codimension 2 smoothable}
\def\CD2{CD2}
\def\O{{\mathcal O}}
\def\Q{{\mathcal Q}}
\def\C{{\mathcal C}}
\def\L{{\mathcal L}}
\def\P{{\mathbb P}}
\def\I{{\mathcal I}}

\def\A{{\mathcal A}}
\def\F{{\mathcal F}}
\def\E{{\mathcal E}}
\def\G{{\mathcal G}}
\def\K{{\mathcal K}}
\def\N{{\mathcal N}}

\def\Z{{\mathbb Z}}

\def\B{{\mathcal B}}

\def\Pthree{{\mathbb P}^3}
\def\Pn{{\mathbb P}^n}

\def\Hom{\mathop{\mathcal Hom}}
\def\Ext{\mathscr{E}\text{\kern - 3pt {\calligra\large xt}}\,} 

\def\Spec{\mathop{\rm Spec}}
\def\rk{\mathop{\rm rank}}
\def\codim{\mathop {\rm codim}}
\def\supp{\mathop {\rm Supp}}
\def\mod{\mathop{\rm mod}}
\def\m{\mathop{\rm m}}

\def\supp{\mathop{\rm Supp}}
\def\sing{\mathop{\rm Sing}}
\def\rank{\mathop{\rm rank}}

\def\deg{\mathop{\rm deg}}

\def\cok{\mathop{\rm Coker}}
\def\ker{\mathop{\rm Ker}}

\def\k{\mathop{\rm k}}

\title{Smoothing surfaces on fourfolds}

\author{Scott Nollet}

\address{Department of Mathematics, Texas Christian University, Fort Worth, TX 76129}

\email{s.nollet@tcu.edu}

\author{A. P. Rao}

\address{Department of Mathematics, University of Missouri, Saint Louis, MO 63121}

\email{raoa@umsl.edu}


\begin{document}
\bibliographystyle{plain}

\begin{abstract}
If $\E, \F$ are vector bundles of ranks $r-1,r$ on a smooth fourfold $X$ and $\Hom(\E,\F)$ is 
globally generated, it is well known that the general map $\phi: \E \to \F$ is injective and drops rank along a smooth surface. 
Chang improved on this with a filtered Bertini theorem. 
We strengthen these results by proving variants in which (a) $\F$ is not a vector bundle 
and (b) $\Hom(\E,\F)$ is not globally generated.
As an application, we give examples of even linkage classes of surfaces on $\mathbb P^4$ in which 
all integral surfaces are smoothable, including the linkage classes associated with the Horrocks-Mumford surface. 
\end{abstract}

\maketitle

\section{Introduction}
Smoothing results are useful in algebraic geometry, as seen in the many applications of the Bertini theorems \cite{jouanolou}. 
A classical theorem says that if $\E, \F$ are vector bundles of ranks $r-1, r$ on a smooth variety $X$ and $\Hom(\E, \F)$ globally generated, then the general map $\phi: \E \to F$ is injective and if not  locally split, drops rank 
along a codimension $2$ subvariety $Y \subset X$ which is smooth away from a set of codimension $\geq 4$ in $Y$ \cite{kleiman}. 
Chang substantially refined this result with her filtered Bertini theorem \cite{C}. 
To state it, suppose that $0 = \E_0 \subset \E_1 \subset \dots \E_n = \E$ and 
$0 = \F_0 \subset \F_1 \subset \dots \F_n = \F$ are filtrations by subbundles and define 
\begin{equation}\label{FBdata}
\left\{\begin{array}{l} \alpha_i = \rk \F_i - \rk \E_i \text{ for } i < n \\
\B = \{\phi \in \Hom(\E, \F): \phi (\E_i) \subset \F_i\} \subset \Hom (\E,\F).
\end{array}\right.
\end{equation}

\begin{thm}\label{FB} If $\B$ is globally generated, then the general map 
$\phi: \E\ \to \F$ drops rank along $Y$ of codimension two (if non-empty) and 
$\codim_Y \sing Y \geq \min\{2 \alpha_i-1, \alpha_i+2,4\}$. 
\end{thm}

The lower bound in Theorem \ref{FB} is the expected codimension. 
When $\dim X \leq 4$, this says that $Y$ is smooth if $\alpha_i \geq 2$ for each $i < n$ 
(when $\dim X = 5$, we need $\alpha_i \geq 3$). 
Motivated by the liaison theory of the Horrocks-Mumford bundle \cite{HM}, we aim to extend 
Theorem \ref{FB} to situations where $\dim X \leq 4$ and 
(a) $\F$ is not a vector bundle or (b) $\B$ is not globally generated. 

We use Fitting schemes to classify rank $r$ sheaves $\F$ on a smooth variety $X$ for which there are locally non-split 
maps $\O^{r-1} \to \F$ dropping rank along a 
smooth subvariety of codimension two (Proposition \ref{classify}), 
calling the resulting sheaves \textit{codimension $2$ smoothable} (\CD2 for short). 
These generalize the curvilinear sheaves on $\mathbb P^3$ introduced by Hartshorne and Hirschowitz \cite{HH}. 
Let $\F$ be a rank $r$ \CD2 reflexive sheaf whose singular scheme $\sing \F$ has integral curve components and $\E$ be a rank $r-1$ vector bundle. 
Suppose $\E$ has a locally split filtration $\E_i$ by subbundles, and $\F$ has a locally split filtration $\F_i$ by \CD2 reflexive sheaves.  
Define $\B$ and $\alpha_i$ as in (\ref{FBdata}). 

\begin{thm}\label{main1}
Suppose $\dim X = 3$ or $4$.
If $\B$ is globally generated and $\alpha_i \geq 2$ for $i < n$, then $\phi: \E \to \F$ drops rank along 
smooth $Y \subset X$ of codimension $2$ for general $\phi$, if $Y \neq \emptyset$.
\end{thm}

When $n=1$ and $X = \mathbb P^3$ we recover \cite[Theorem 3.2]{HH}. 

Our second result gives a variant of Theorem \ref{FB} when $\B$ is not globally generated. 
It is harder to make an abstract statement, so we take $X = \mathbb P^d$ with $d \leq 4$, 
$\E = \oplus \O (-a_i)$ and $\F = \oplus \O (b_j) \oplus \G$, where $\G$ is a rank $r$ vector bundle having a space of 
sections $V \subset H^0 (\G)$ for which the evaluation map $V \otimes \O_X \to \G$ has cokernel $\Q$ which 
is generically a line bundle on a smooth curve. 
Assuming $H^0(\G(-1))=0$, we define a \em{canonical filtration }\em $\E_i, \F_i$ on $\E$ and $\F$ based on \cite[Example 2.1]{C}. 
The corresponding sheaf $\B$ in (\ref{FBdata}) need not be globally generated, but even so we obtain smoothing: 

\begin{thm}\label{main2}
Suppose $X = \mathbb P^d$ with $d = 3$ or $4$. 
If $\alpha_i \geq 2$ for $i < n$, then $\phi: \E \to \F$ drops rank along a 
smooth subvariety $Y \subset X$ of codimension $2$ for general $\phi$, if $Y \neq \emptyset$. 
\end{thm}

When $\E = \O^{r-1}$ and $\F = \G$, the general map $\phi: \E \to \F$ 
drops rank along a smooth subvariety of codimension two, recovering the fact that a general section of the Horrocks-Mumford 
bundle vanishes along a smooth surface \cite[Theorem 5.1]{H}. 
Our Theorem \ref{notgg} proves this more generally when $\G$ is \CD2 reflexive, but state it here for $\G$ a vector bundle to make the statement cleaner. 

\subsection{Applications to linkage theory}
Linkage theory \cite{JM,PS} treats general locally Cohen-Macaulay subschemes of $\mathbb P^d$ of codimension $2$, but one is often interested in which subschemes $Z \subset \mathbb P^d$ can be deformed to a smooth variety within its even linkage class $\L$. 
By \cite{R2}, there is a vector bundle $\N_0$ with $H^1_* (\N_0^\vee)=0$ corresponding to $\L$ for which each $Y \in \L$ has a resolution of the form  
\[
\begin{array}{ccccccccc}
0 & \to & \oplus \O (-a_i) & \stackrel{\phi}{\to} & \oplus \O (-b_j) \oplus \N_0 & \to & \I_Y (t) & \to & 0 \\
& & || & & || & & & & \\
& & \E & \stackrel{\phi}{\to} & \F & & & & 
\end{array}
\]
with $t \in \Z$, so smoothing becomes a question of whether a general map $\phi$ drops rank along a smooth subvariety of codimension $2$. 
Chang \cite{C,C2,C3} applied Theorem \ref{FB} to these resolutions to classify smooth arithmetically Buchsbaum codimension $2$ subvarieties in $\mathbb P^d$ 
for $d \leq 5$: none exist for $d \geq 6$, as predicted by Hartshorne's conjecture \cite{H1}. 
Building on work of Sauer \cite{sauer}, Steffen used Theorem \ref{FB} to classify codimension $2$ smooth connected ACM subvarieties in $\mathbb P^n$ \cite{steffen}. 
An interesting feature of these examples is that 
every integral curve in ACM or arithmetically Buchsbaum linkage classes on $\mathbb P^3$ is 
smoothable \cite{C,GP1,PR}. 
The same holds for ACM or arithmetically Buchsbaum linkage classes of surfaces on $\mathbb P^4$ \cite{C,Nsecant}. 
This makes it easy to write down the deformation classes having a smooth variety because there is a numerical criterion for integrality in these classes \cite{Nsecant,Nint}. 

Our work here is motivated by the linkage theory of the Horrocks-Mumford 
bundle $\F_{HM}$ \cite{HM}. 
It is the only known indecomposable rank two vector bundle on $\mathbb P^4_k$ if char $k=0$, though others have been discovered when char $k = p > 0$ \cite{MK,MKPR}. 
The bundle $\F_{HM}$ is not globally generated, but has a space of sections for which the 
cokernel of the evaluation map is a line bundle on a smooth curve $L$ which is a union of $25$ disjoint lines; 
a general such section vanishes along an abelian surface $X_0$ which is minimal for its even linkage class $\L$. 
We show that $\F_{HM}$ is a quotient of the rank $7$ vector bundle $\N_0$ corresponding to 
$\L$ via the correspondence \cite{R2} and use Theorem \ref{main2} to show that every integral surface in $\L$ is smoothable (Example \ref{HMeven}).
The bundle $\N_0^*$ corresponding to the odd Horrocks-Mumford linkage class $\L^*$ 
has rank $17$: we construct a rank $2$ quotient sheaf $\A$ of $\N_0^*$ which is \CD2 reflexive with singular scheme precisely the curve $L$ consisting of $25$ lines and use Theorem \ref{main1} to show that every integral surface in $\L^*$ is smoothable (Example \ref{HModd}). 

Syzygy bundles provide another interesting example of even linkage classes. The kernel $\N_0$ of a surjection $\oplus_{i=1}^4 \O_{\mathbb P^3} (-d_i) \to \O_{\mathbb P^3}$ determines an even linkage class of of curves on $\mathbb P^3$. Martin-Deschamps and Perrin completely worked out the smoothable classes in these cases \cite{MDP2} and these are typically not the same as the integral elements \cite{Nfix}, the smallest numerical case being Hartshorne's example of an integral curve not smoothable in the Hilbert scheme \cite{H}. 
Similarly the kernel $\N_0$ of a surjection $\oplus_{i=1}^5 \O_{\mathbb P^4} (-d_i) \to \O_{\mathbb P^4}$ 
gives an even linkage class of surfaces on $\mathbb P^4$. 
In $\S 4$ we show that all integral elements are smoothable in these classes when all the $d_i$ are the same (Example \ref{surface}), but in general we can expect a situation as complicated as for curves on $\mathbb P^3$, so we pose the following.  

\begin{ques}\label{Q2}{\em
Let $\N_0$ be the kernel of a surjection $\oplus_{i=1}^5 \O (-d_i) \to \O$ on $\mathbb P^4$. 
Which members of the corresponding even linkage class $\L$ deform to smooth or integral varieties?
\em}\end{ques}

This work is organized as follows. 
In Section 2 we use Fitting schemes to classify sheaves whose local quotient by a vector 
bundle is an ideal sheaf of a smooth codimension two 
subvariety and prove Theorem \ref{main1}. 
In Section 3 we consider reflexive sheaves 
with spaces of sections that don't generate, but whose cokernel of the evaluation map behaves well. 
The main result is Theorem \ref{notgg}, 
which generalizes Theorems \ref{main1} and \ref{main2} when $X = \mathbb P^3$ or $X = \mathbb P^4$. 
In Section 4 we give applications to smoothing members in even linkage 
classes of curves in $\mathbb P^3$ and surfaces on $\mathbb P^4$, including an explanation of the linkage theory of the Horrocks-Mumford surface.   

\section{Reflexive sheaves and stratification by rank}

We use Fitting ideals to classify the coherent sheaves $\F$ on a smooth variety $X$ which are locally 
the extension of a vector bundle and an ideal sheaf of a smooth codimension two subvariety;  
the resulting sheaves are \cd2 (abbreviated \CD2).  
When $X = \mathbb P^3$, these are the curvilinear sheaves introduced by Hartshorne and Hirschowitz \cite{HH} and studied by Martin-Deschamps and Perrin \cite{MDP2}. 
We show that if $\F$ is a rank $r$ reflexive \CD2 on a smooth fourfold $X$ and $\sing \F$ has integral curve components, 
then sufficiently general maps $\phi: \mathcal E \to \F$ with $\E$ a rank $(r-1)$-bundle drop rank along a smooth surface. 
We give a variant of Chang's filtered Bertini theorem \cite{C} for these sheaves. 

A coherent sheaf $\F$ on a smooth variety $X$ has a local presentation 
\begin{equation}\label{pres}
\O^n_U \stackrel{u}{\to} \O^m_U \to \F_U \to 0
\end{equation}
on an open affine $U \subset X$. 
The $i$th Fitting scheme $S_i (\F)$ has ideal generated by the $(m-i+1)$-minors of the matrix for $u$. 
Since this ideal is independent of the presentation \cite[$\S 20.2$]{E}, $S_i (\F)$ is well defined and 
is set-theoretically the locus where $\rk u \leq m-i$, or equivalently $\rk \F \geq i$. 
Since $\rm{rank} \; \F = r$, then $S_i (\F) = X$ for $i \leq r$ and we define 
$\sing (\F) = S_{r+1} (\F)$, the {\it singular scheme} of $\F$, the closed subscheme where $\F$ is not a vector bundle. 
A closed subscheme $Z \subset X$ is {\it \cd2} (CD2 for short) if $Z$ has local 
embedding dimension at most $\dim X - 2$, or equivalently $Z$ locally lies on a smooth subvariety of codimension $2$. 

\begin{prop}\label{obvy}
Let $\F$ be a coherent sheaf on a smooth variety $X$ and let $\mathcal P$ be locally free of rank $k$. 
\begin{enumerate}
\item[(a)] If $\mathcal P \to \F \to \F^\prime \to 0$ is exact, then 
$S_{k+i} (\F) \subset S_{i} (\F^\prime)$. 
\item[(b)] If $0 \to \F^\prime \to \F \to \mathcal P \to 0$ is exact, then 
$S_{k+i} (\F) = S_{i} (\F^\prime)$.
\item[(c)] If $S_i (\F)$ is \CD2, then $S_{i+1} (\F)$ is empty.
\end{enumerate}
\end{prop}

\begin{proof}
(a) Suppose $\F$ has local presentation (\ref{pres}). Then $S_{k+i} (\F)$ is empty for $k+i > m$ because $\F$ is locally generated by $m$ elements, 
so we may assume $k+i \leq m$. Working on an open affine $U$ where 
$\mathcal P_U \cong \O_U^k$, we obtain a presentation $\O^n \oplus \O^k \stackrel{u^\prime}{\to} \O^m \to \F^\prime$ 
where $u^\prime = [u,a]$ and $a$ is a $k \times m$ matrix. Since $m-i+1 > k$, each 
($m-i+1$)-minor of $u^\prime$ expands in terms of ($m-k-i+1$)-minors from $u$, which shows that 
$S_{k+i} (\F) \subset S_i (\F^\prime)$ scheme-theoretically. 

(b) Locally $\F_U \cong \F^\prime_U \oplus \O_U^k$, so a local presentation $\O_U^n \stackrel{u^\prime}{\to} \O_U^m \to \F^\prime \to 0$ 
yields $\O_U^n \stackrel{u}{\to} \O_U^{m+k} \to \F \to 0$ with $u = \left[\begin{array}{c} u^\prime \\ 0 \end{array} \right]$ and the 
minors generating the ideal of $S_{k+i} (\F)$ are equal to those generating the ideal of $S_i (\F^\prime)$. 

(c) To compute the Fitting ideals of $\F$ at $p\in X$, we may assume the local presentation 
$u_p: \O_p^n \to \O_p^m$ is replaced by a minimal presentation, so that each entry of $u$ is in $\m_p$. 
Thus if $p \in S_{i+1}(\F)$, then all $m-i$ minors of $u_p$ belong to $\m_p$, hence all $m-i+1$ minors belong to $\m_p^2$, 
which implies that $S_i(\F)$ cannot be \CD2 at $p$. See also \cite[II, Cor. 1.7]{MDP2}.
 
\end{proof}

\begin{defn}\label{CD2}{\em A rank $r$ sheaf $\F$ on $X$ is {\em \cd2} (\CD2 for short) if $\F$ is torsion 
free and $\sing (\F) = S_{r+1} (\F)$ is a \CD2 scheme.}
\end{defn}
This extends the notion of curvilinear sheaves on $X=\mathbb P^3$ introduced by Hartshorne and Hirschowitz \cite{HH}. 
We extend \cite[II, Proposition 3.6]{MDP2} to higher dimension as follows. 

\begin{prop}\label{classify}
Let $\F$ be a rank $r$ sheaf on a smooth variety $X$ with $\dim X \geq 2$. Then the following are equivalent:
\begin{enumerate}
\item $\F$ is a \CD2 sheaf. 
\item For each $p \in X$, the stalk $\F_p$ satisfies one of the following:
\begin{enumerate}
\item $\F_p \cong \O_p^r$.
\item $\F_p$ is the cokernel of a map $\O_p \stackrel{(x,y,f_3, \dots, f_{r+1})}{\to} \O_p^{r+1}$ with 
$x,y \in \O_p$ part of a regular system of parameters and $f_i \in \m_p$ for each $i$. 
\end{enumerate}
\item For each $p \in X$, there is an exact sequence $0 \to \O_p^{r-1} \to \F_x \to \I_{S,p} \to 0$ with $S$ a smooth germ at $p$ of codimension $2$. 
\end{enumerate}
\end{prop} 

\begin{proof} 

$(1) \Rightarrow (2):$ If $p \not \in S_{r+1} (\F)$, then $\F$ is locally free at $p$ giving (a), so we may assume $p \in S_{r+1} (\F)$. 
Then $S_{r+2} (\F) = \emptyset$ by Proposition \ref{obvy} (c), so $\F$ has rank $r+1$ at $p$ and a local presentation
\[
\O^m_p \stackrel{u}{\to} \O^{r+1}_p \to \F_p \to 0,
\]
with entries of the matrix $u$ generating the ideal for $S_{r+1} (\F)$ and the $2 \times 2$ minors of $u$ vanishing on a neighborhood of $p$, 
hence equal to $0$ in $\m_p$. 
We may assume that $u_{1,1}=x \in \m_p - \m_p^2$. 
Suppose that $x$ does not divide $u_{1,j}$ for some $j > 1$. Then $x$ divides $u_{k,1}$ for all $k$ due to the vanishing $2 \times 2$ minors, 
hence the first column of $u$ has the form $x b$ with $b = (1, b_2, \dots, b_{r+1})^t$. 
Since $x b$ maps to zero in $\F_p$ which is torsion free, the image of $b$ in $\F_p$ is zero, hence is in the image of $u$, but 
this is impossible since all entries of $u$ lie in $\m_p$. Therefore $x$ divides $u_{1,j}$ for each $j$, so we can write 
$u_{1,j} = x w_j$ with $w_1 = 1$ and $v_j = u_{j,1}$. 
The vanishing of $2 \times 2$ minors yields $u_{i,j} = v_i w_j$ for $i, j > 1$, so each column of $u$ is a multiple of the first 
column and the image of $u$ is the span of the first column, thus we may assume $m=1$. Since the entries of $u$ generate the ideal 
of $S_{r+1} (\F)$, we may assume $u_{2,1} = y$ where $x,y$ are part of a regular system of parameters for $\m_p$, giving possibility (b). 

$(2) \Rightarrow (3):$ Clear in case (a) by taking $p \not \in S$. In case (b), let $\pi: \O_p^{r+1} \to \O_p^2$ be the projection onto the first two factors. Then $(x,y)$ defines a smooth 
codimension two subvariety $S$ locally at $p$ and we apply the snake lemma to 
\[\begin{array}{ccccccccc}
0 & \to & \O_p & \stackrel{(x,y,f_3, \dots, f_{r+1})}{\to} & \O_p^{r+1} & \to & \F_p & \to & 0 \\
 & & \downarrow & &  \downarrow \pi & & \downarrow &  & \\
0 & \to & \O_p & \stackrel{(x,y)}{\to} & \O_p^{2} & \to & \I_{S,p} & \to & 0.
\end{array}\]

$(3) \Rightarrow (1):$ Follows from Proposition \ref{obvy} (a).
\end{proof}

\begin{rmk}\label{consequence}{\em Let $\F$ be a \CD2 sheaf of rank $r$ on $X$ as in Proposition \ref{classify}.
\begin{enumerate}
\item[(a)] If $r=1$, the embedding $\F \hookrightarrow \F^{\vee\vee} = \mathcal L$ shows that $\F \cong \I_S \otimes \L$ with $S$ smooth of codimension two and $\L$ a line bundle. 

\item[(b)] When $\F$ is \CD2 of rank $r$ and $p \in \sing \F$, the ideal of $\sing \F$ at $p$ is 
generated by $x,y,f_3, \dots f_{r+1}$ appearing in Proposition \ref{classify} (2b). 
If some $f_i \not \in (x,y)$ in Proposition \ref{classify} (2b), then $\F_p$ is reflexive by Lemma \ref{MDPremark2.4} because $\codim \sing \F \geq 3$. 
On the other hand, if all $f_i \in (x,y)$, then we can change basis so they become $0$, in which case 
$\F_p \cong \O_p^{r-1} \oplus \I_{S,p}$ with $S$ smooth of codimension two defined by $(x,y)$, hence $\F_p$ is not reflexive. 
 
\item[(c)] When $\dim X = 3$, the local ideal $(x,y,f_3, \dots, f_{r+1})$ of $p \in \sing \F$ can be written $(x,y,z^n)$ with $n \geq 1$ and we 
recover \cite[II, Proposition 3.6]{MDP2}. 

\item[(d)] When $\dim X =4$, the local ideal can be written $(x,y,f_3, \dots, f_{r+1})$ with $x,y,z,w$ local parameters 
and $f_i \in (z,w)$. For example, we can define a \CD2 reflexive rank $3$ reflexive sheaf on $X = \mathbb A^4$ by 
$0 \to \O \stackrel{(x,y,z^2w^2,zw^3)}{\to} \O^4 \to \F \to 0$. 
The singular scheme $\sing (\F)$ is the union of the line $x=y=z=0$, the double line $x=y=w^2=0$ and an embedded point supported  at the origin. 
For our smoothing results, we will avoid such non-reduced curves.
\end{enumerate}
\em}\end{rmk}

The next result helps to identify reflexive quotients of reflexive sheaves. 

\begin{lem}\label{MDPremark2.4}
Suppose $0 \to \mathcal P \to \E \to \F \to 0$ is exact with $\mathcal P$ locally free and $\E$ reflexive. 
Then $\F$ is reflexive if and only if $\codim \sing \F \geq 3$. 
\end{lem}

\begin{proof} $\Rightarrow:$ If $\F$ is reflexive, then $\codim \sing \F \geq 3$ by \cite[Corollary 1.4]{SRS}.

$\Leftarrow:$ 
First observe that $\F$ is torsion free, or equivalently that $H^0_x (\F_x) = 0$ for each non-generic point $x \in X$. 
This is clear if $\dim \O_x \leq 2$ because $\F_x$ is a free $\O_x$-module. 
If $\dim \O_x > 2$, then 
$H^0_x (\E_x) = 0$ because $\E$ is torsion free and $H^1_x (\mathcal P_x)=0$ because $\rm{depth} \; \mathcal P_x = \dim \O_x > 1$, 
so the long exact local cohomology sequence gives $H^0_x (\F_x) = 0$. 
It remains to show that $\rm{depth} \; \F_x \geq 2$ whenever $\dim \O_x \geq 2$ \cite[Prop. 1.3]{SRS}. 
This is clear if $\dim \O_x = 2$ because $\F_x$ is free. 
If $\dim \O_x > 2$, then $\rm{depth} \; \E_x \geq 2$ because $\E_x$ is reflexive, hence 
$H^i_x (\E_x)=0$ for $i < 2$. Also $H^i_x (\mathcal P_x)=0$ for $i < \dim \O_x$, so the long exact local cohomology sequence 
shows that $H^i_x (\F_x)=0$ for $i < 2$, therefore $\rm{depth} \; \F_x \geq 2$. 
\end{proof}

If there is an exact sequence 
\begin{equation}\label{eins}
0 \to \E \stackrel{\phi}{\to} \F \to \I_Y \otimes \L \to 0
\end{equation}
with $\E$ locally free, $\L$ a line bundle, and $Y$ smooth of codimension two, then $\F$ is a \CD2 by Proposition \ref{classify} (3) 
and we say that $\cok \phi$ is a \textit{twisted ideal sheaf} of $Y$. 
We will show in Theorem \ref{rank2} that if $\rank \E = \rank \F -1$, $\Hom (\E, \F)$ is globally generated and $\dim X \leq 4$, then $\cok \phi$ is the twisted ideal sheaf of a codimension two smooth 
subscheme. 
We will repeatedly use the following in our dimension counting arguments. 

\begin{lem}\label{count}
Let $M_{a,b} (k) \cong \mathbb A^{ab}$ be the space of $a \times b$ matrices over a field $k$ with $a \leq b$. 
If $c \leq a$,  then the space $M_c \subset M_{a,b} (k)$ of matrices having rank $\leq c$ is a 
subvariety of codimension $(a-c)(b-c)$ with singular locus $\sing M_c = M_{c-1}$. 
\end{lem}

\begin{proof}
See \cite[Teorema 2.1]{O}. 
\end{proof}

\begin{thm}\label{rank2}
Let $\F$ be a rank $r$ reflexive \CD2 sheaf on a smooth fourfold $X$ with 
$\sing \F$ having all curve components integral. Let $\E$ be a rank $k < r$ vector bundle and $V \subset H^0(\Hom (\E, \F))$ a 
finite dimensional vector space of sections that globally generate. Then the general map $\phi: \E \to \F$ is injective, 
let $\overline \F = \cok \phi$ be the cokernel.
\begin{enumerate}
\item[(a)] If $k=r-1$, then $\overline \F = \cok \phi$ is the twisted ideal sheaf of a smooth surface. 
\item[(b)] If $k < r -1$, then $\overline \F = \cok \phi$ is a reflexive \CD2 sheaf 
with curve components of $\sing \overline \F$ integral.  
\end{enumerate}
\end{thm}
 
\begin{proof}
Let $U = X - \sing \F$ so that $\F_U$ is locally free. 
Since $V$ generates $\Hom(\E_U, \F_U)$, the general map $\phi: \E_U \to \F_U$ is injective and drops rank along smooth 
$Y \subset U$ of codimension $r-k+1$ by Theorem \ref{FB}.
Since $Y \subset U$ is locally defined by $k \times k$ minors of the matrix representing $\phi$, 
$\sing \overline\F_U =Y$ by definition and $\overline \F_U$ is \CD2. 
Thus the general map $\phi: \E \to \F$ is injective with cokernel $\overline \F$ singular along $Y$ and  
possibly other points of $\sing \F$. 
To complete the proof, we use dimension counting arguments to show that $\overline \F$ behaves as required along $\sing \F$, 
which has dimension $\leq 1$ by \cite[Cor. 1.4]{SRS}. 

For $p \in \sing \F$, Proposition \ref{classify} gives 
a local resolution $0 \to \O_p \stackrel{u}{\to} \O_p^{r+1} \to \F_p \to 0$ where $u = [x,y,f_3, \dots, f_{r+1}]^T$, 
$x,y \in \O_p$ are part of a sequence of parameters, and not all $f_i \in (x,y)$ by Remark \ref{consequence} (b). 
A map $\phi \in V$ localizes to $\phi_p:\E_p \to \F_p$, which lifts to a map 
$a: \O_p^k \to \O_p^{r+1}$ and  
gives a resolution 
$0 \to \O_p \oplus \O_p^k \xrightarrow{[u,a]} \O_p^{r+1} \to \overline \F_p \to 0$.
Furthermore, the map $\phi_p \otimes \k(p)$ is given by the matrix 
$\bar a = a \mod \m_p \in M_{r+1,k}(k(p))$. 
Since $V$ generates $\Hom(\E, \F)$, the map $V \to \Hom(\E(p), \F(p))$ is onto. 
The subspace of matrices $\bar a$ of rank $\leq k-1$ has codimension $r-k+2$ in $M_{r+1,k}(k(p))$ by Lemma \ref{count}, 
hence also its pre-image in $V$. 
This shows that $\B =\{(\phi , p) | \rank \bar a \leq k-1\} \subset V \times \sing \F$ has dimension 
$\dim V - (r-k+2) + 1$, hence cannot dominate $V$. 
Therefore the general $\phi \in V$ has the property that $\phi_p$ has rank $k$ modulo $\m_p$ at each point $p \in \sing \F$.

First we prove (a), so let $k=r-1$. 
We distinguish between the smooth points on an integral curve component of $\sing \F$ and the finite set of singular or isolated points. 
At a smooth point $p$, $u = [x,y,z,0,0\dots, 0]^T$ resolves $\F_p$. 
Let $a_3$ be the $(k-1)\times k$ submatrix of $a$ obtained by deleting the top $3$ rows. 
By Lemma \ref{count}, the space of matrices $\bar a$ with $\rank \bar a_3 \leq k-2$ has codimension $2$, 
hence $\B_1 = \{\phi , p) | \rank \bar a_3 \leq k-2\} \subset V \times \sing \F$ has dimension $\dim V -2 +1$ and so the general map 
$\phi \in V$ yields $a_3$ of rank $k-1$ at all smooth points of $\sing \F$. 
At any smooth point, after choosing bases, we may assume that 
\[
[u,a] = \begin{bmatrix} x & b & 0 & \dots & 0 \\ y & c & 0 &  \dots & 0 \\ z & d & 0 &\dots &0 \\ 0 & 0 & 1 &\dots & 0 \\ \vdots & \vdots & \vdots & \vdots &\vdots \\ 0 & 0 & 0 & \dots & 1 \end{bmatrix}.
\]
In the previous paragraph we saw that $a$ has rank $k$ modulo $\m_p$ for general $\phi$, so one of $b,c,d$ is a unit in $\O_p$. 
If $d$ is a unit, then $\sing \overline\F_p$ is defined by the ideal  $(dx-bz, dy-cz)$, which defines a smooth local surface.

Now consider the finite subset of singular and isolated points, where $\F_p$ is resolved by $u = [x,y, f_3, \dots, f_{r+1}]^T$ and $k=r-1$. 
Let $a_2$ be the $(r-1)\times k$ submatrix of $a$ obtained by removing the top two rows. 
The subspace of all matrices $\bar a$ for which the $k \times k$ submatrix $\bar a_2$ has rank $\leq k-1$ is of codimension $1$ by Lemma \ref{count}. 
Hence the general map $\phi \in V$ yields $\phi_p$ for which $\bar a_2$ has rank $k$ at each of these points, 
so $a_2$ is a nonsingular $k\times k$ matrix, with a unit $d$ for determinant. Since elementary row operations can make the top two rows of $a$ equal to zero, one sees that the Fitting ideal of $\overline \F_p$ is just $(dx + \text{terms in}f_i, dy + \text{terms in}f_i)$, 
which again defines a smooth local surface.

Now we prove (b), so let $k  < r-1$. Since $\sing \overline \F \subset Y \cup \sing \F$ and $\dim Y \leq 1$, 
$\overline \F$ is reflexive by Lemma \ref{MDPremark2.4}. 
The subspace of matrices $\bar a$ such that $\rank \bar a_2 \leq k-1$ is of codimension $(r-1) - (k-1)$ by Lemma \ref{count}, so $\B_2 = \{ (\phi, p) | \rank \bar a_2 \leq k-1\} \subset V \times \sing \F$ has dimension at most $\dim V -(r-k) +1$. 
Since $r-k > 1$, $\B_2$ does not dominate $V$ and the general $\phi \in V$ gives rise to a matrix $\phi_p = a$ for which $a_2$ has a $k\times k$ minor which is a unit $d \in \O_p$. 
When we compute the Fitting ideal of $\overline \F_p$, the $k+1$ minor of $[u,a]$ that uses the first row and the rows of this $k\times k$ minor works out to $dx + \text {terms in } f_i$. 
Likewise the second row gives $dy + \text {terms in } f_i$. 
Since $dx, dy$ are part of a system of parameters for $\O_p$ and $f_i \not \in (x,y)$, it follows that $\overline \F$ is CD2 at $p$. 

It remains to show that the curve components of $\overline \F$ are integral. 
Letting $p \in \sing \F$ be a smooth point on a curve component, $\sing F$ has an ideal of the form $(x,y,z)$ with $x,y,z$ part of a regular sequence of parameters and we may assume that $u = [x,y,z,0,0\dots, 0]^T$ resolves $\F_p$. 
The matrices $a$ for which $\rank \bar a_3 \leq k-1$ has codimension $(r-2) - (k-1)$, so the set 
$\B_3 = \{(\phi,p): \rank \bar a_3 \leq k-1\} \subset V \times \sing \F$ has dimension $\dim V - (r-k-1)+1$.  
If $k < r-2$, this dimension is less than $\dim V$, and hence $\B_3$ does not dominate $V$ and the general $\phi \in V$ has the property that at each smooth integral 
curve point  of $\sing \F$, $\phi(p)$ has the corresponding $\bar a_3$ of rank $k$. 
If $k = r-2$, then $\B_3$ has the same dimension as $V$, and so the fibre over the general $\phi$ in $V$ can have only finitely many 
points in $\B_3$. This means that for the general $\phi$, all but finitely many of the points in $\sing \F$ will yield  a $\phi(p)$ with 
$\rank \bar a_3 = k$. 
Therefore at the general point $p$ of a curve component of $\sing \F$, there is a $k \times k$ minor of $a_3$ with 
determinant $d$ a unit and the Fitting ideal of $\overline \F_p$ will contain $dx, dy, dz$, so at these points, so that 
$\sing \F_p = \sing \overline \F_p$. 
This proves part (b). 
\end{proof}

We strengthen Theorem \ref{rank2} for later use. 

\begin{cor}\label{avoidance}
In Theorem \ref{rank2}, let $A \subset X$ be closed with $A \cap \sing \F = \emptyset$ and $\dim A \leq 1$. 
Then for general $\phi: \E \to F$, the cokernel $\overline \F$ is locally free along $A$. 
\end{cor}

\begin{proof}
Give $A$ the reduced scheme structure, let $\{p_1, \dots, p_m\}$ be the isolated points and singular 
points of the curve components of $A$, and $A_1, \dots, A_n$ the irreducible smooth curve 
components of $A - \{p_1, \dots, p_m\}$. 
The restriction $\F_{A_1}$ is a vector bundle and $V$ generates the sheaf $\Hom (\E_{A_1}, \F_{A_1})$, so 
by Theorem \ref{FB}, the general map $\phi \in V$ has restriction $\phi_{A_1}$ has empty degeneracy locus, 
meaning that $\overline \F$ is locally free along ${A_1}$: let $V_1 \subset V$ be a Zariski open set of such $\phi$. 
Similarly form Zariski open sets $V_2, \dots V_n$ for each $A_2, \dots A_n$ and $V_{n+1}, \dots V_{n+m}$ 
for each $p_1, \dots p_m$. 
Theorem \ref{rank2} gives a Zariski open set $V_{n+m+1}$ of maps $\phi$ for which 
$\overline \F$ is reflexive \CD2 reflexive with curve components of $\sing \overline \F$ integral. 
Taking $\phi \in \cap_{k=1}^{n+m+1} V_k$ proves the corollary. 
\end{proof}

Now we give a filtered version of Theorem \ref{rank2} (b) which generalizes Theorem \ref{FB} to \CD2 reflexive sheaves 
when $\dim X \leq 4$. 
Let $X$ be a smooth fourfold, 
$\E$ a vector bundle on $X$ with a split filtration by subbundles 
$0 = \E_0 \subset \E_1 \subset \dots \subset \E_n =\E$ and let 
$\F$ be a \CD2 reflexive sheaf on $X$ have a locally split filtration by sheaves 
$0 = \F_0 \subset \F_1 \subset \dots \subset \F_n = \F$. 
Set $\alpha_i = \rk \F_i - \rk \E_i$ for $1 \leq i < n$, $\alpha = \rk \F - \rk \E = 1$ 
and let 
\[
\B = \{\phi \in \Hom (\E, \F): \phi (\E_i) \subset \F_i \text{ for each } 1 \leq i \leq n \}.
\] 

\begin{rmk}\label{filtration}{\em We note two consequences of our hypotheses on the filtrations. 

(a) The splitting of the filtration on $\E$ induces a splitting of $\B$. 
Define $\C_i$ by the split exact sequences $0 \to \mathcal E_i \to \mathcal E_{i+1} \to \mathcal C_{i+1} \to 0$. 
Set $\B_1 = \Hom (\E_1, \F_1) \cong \E_1^\vee \otimes \F_1 \subset \E_1^\vee \otimes \F_2$ and let 
$\pi: \E_2^\vee \otimes \F_2 \to \E_1^\vee \otimes \F_2$ be the natural surjection. 
Take $\B_2 = \pi^{-1} (\B_1)$, the set of homomorphisms $\phi: \E_2 \to \F_2$ such that $\phi (\E_1) \subset \F_1$. 
The splitting of the bottom row of 
\[
\begin{array}{ccccccccc}
0 & \to & \mathcal C_2^\vee \otimes \F_2 & \to & \B_2 & \to & \B_1 & \to & 0 \\
& & || & & \cap & & \cap & & \\
0 & \to & \mathcal C_2^\vee \otimes \F_2 & \to & \E_2^\vee \otimes \F_2 & \to & \E_1^\vee \otimes \F_2 & \to & 0
\end{array}
\]
shows that the top row splits as well, giving $\B_2 \cong \B_1 \oplus (\mathcal C_2^\vee \otimes \F_2)$.
Continuing in this way, we find that $\B \cong \oplus_{i=1}^n (\mathcal C_{i}^\vee \otimes \F_{i})$. 

(b) Let $\Q_i = \F_i/\F_{i-1}$ 
and let the ranks of the sheaves in the 
locally split sequence $0 \to \F_{n-1} \to \F_{n} \to \Q_{n} \to 0$ be $r,s,t$ with $s=r+t$. 
Due to the splitting of stalks at $p \in X$, $(\F_n)_p$ is a free $\O_p$-module if and only if both $(\F_{n-1})_p$ and $(\Q_n)_p$ are. 
On the other hand, if $p \in \sing \F_n$, then the stalk $(\F_n)_p$ is at most $(s+1)$-generated by Proposition \ref{classify}, 
hence $(\F_{n-1})_p$ is at most $r$-generated or 
$(\Q_n)_p$ is at most $t$-generated, so one of $(\F_{n-1})_p$ or $(\Q_n)_p$ is free and the other is \CD2 reflexive  
by Proposition \ref{obvy} (b). 
Therefore $\F_n$ is \CD2 reflexive if and only if both $\F_{n-1}$ and $\Q_n$ are \CD2 reflexive and 
$\sing \F_{n-1} \cap \sing \Q_n = \emptyset$, in which case $\sing \F_n = \sing \F_{n-1} \cup \sing \Q_n$. 
Continuing through the locally split exact sequences, we see that the sheaves $\Q_{i}$ 
are reflexive \CD2 with disjoint singular schemes $C_{i}$ of dimension at most one with integral curve components and 
that $\sing \F_k = \cup_{i=1}^k C_{i}$ for $1 \leq k \leq n$.  
\em}\end{rmk}

\begin{thm}\label{filterreflexive}
Assume $\E$ is a bundle, $\F$ a reflexive \CD2 sheaf with curve components of $\sing \F$ integral, which have split and locally split filtrations as above, and assume 
that $\B$ is generated by a finite dimensional subspace $V \subset H^0 (\B)$. If $\alpha_i \geq 2$ for $1 \leq i < n$, 
then there is an injective map $\phi: \E \to \F$ whose cokernel is the twisted ideal sheaf of a smooth surface. 
If $X$ is projective, this is true of the general map $\phi$. 
\end{thm}

\begin{proof}
There is an isomorphism $\B \cong \oplus_{i=1}^n (\mathcal C_{i}^\vee \otimes \F_{i})$ by Remark \ref{filtration} (a). 
Taking $V_i$ to be the projection of $V$ to $H^0 (\mathcal C_{i}^\vee \otimes \F_{i})$, we may replace 
$V$ with the possibly larger subspace $V_1 \oplus \dots \oplus V_n \subset H^0 (\B)$.
By Remark \ref{filtration} (b), the quotients $\Q_i$ are reflexive \CD2 with disjoint singular schemes. 

We induct on $n \geq 1$.  
Theorem \ref{rank2} (b) covers the case $n=1$, so assume $n > 1$. 
Since $\alpha_1 \geq 2$, the general map $\phi_1: \E_1 \to \F_1$ is injective with 
quotient $\overline \F_1$ a reflexive \CD2 sheaf with $\sing \overline \F_1$ having integral curve components 
and $\overline \F_1$ is locally free along $\cup_{k=2}^n \sing \Q_k$ by Corollary \ref{avoidance}.   
Taking $\overline \F_i = \F_i/\E_1$ and following the locally split exact sequences 
$0 \to \overline \F_{i-1} \to \overline \F_{i} \to \Q_{i} \to 0$ as in Remark \ref{filtration} (b) 
shows that each $\overline \F_k$ is \CD2 reflexive with $\sing \F_k$ having integral curve components. 
The vector spaces $V_i$ generate the quotient sheaves 
$\C_{i}^\vee \otimes \overline \F_{i}$, so we arrive at the situation of the theorem with $n$ one less with 
the filtrations $\overline \F_i$ and $\overline \E_i = \E_i/\E_1$. 
By induction, there is an injective map $\overline \phi: \overline \E \to \overline \F$ whose cokernel is a twisted ideal 
sheaf of a smooth surface and $\overline \phi$ corresponds to $\phi: \E \to \F$ extending $\phi_1$. 

Now suppose that $X$ is projective. Then each map $\phi:\E \to \F$ gives rise to a complex 
$0 \to \E \otimes \L \stackrel{\phi \otimes 1}{\to} \F \otimes \L \stackrel{\phi^\vee \otimes 1}{\to} \O_X$, 
where $\L = \det \E \otimes \det \F^\vee$. 
The set of $\phi \in H^0 (\Hom (\E,\F))$ where the complex is left-exact is open and defines a flat family of subschemes of $X$. 
Since smoothness is an open condition in the Hilbert scheme of a projective variety, 
we obtain a Zariski open set of maps $\phi$ giving rise to a smooth surface. 

\end{proof}

\begin{rmks}\label{extensions}{\em In case $X = \mathbb P^3$, Martin-Deschamps and Perrin have made a deep study 
of maps $\phi: \E = \oplus \O (-a_i) \to \F$ with $\F$ curvilinear, giving necessary and sufficient conditions for when $\phi$ is injective 
with cokernel the twisted ideal sheaf of a smooth curve \cite[Chapters III and IV]{MDP2}. Our counting arguments for 
Theorem \ref{rank2} becomes easier in this setting because $\sing \F$ is discrete, so the dimension counts can be done one fiber at a time. 
Here are the corresponding statements. 

(a) When $\dim X = 3$, our arguments in Theorem \ref{rank2} show that if $\F$ is a rank $r$ reflexive \CD2 sheaf, $\E$ is bundle of rank $k$ 
and $V \subset H^0 (\Hom (\E, \F))$ globally generates, then the general map $\phi: \E \to \F$ is injective, let $\overline \F = \cok \phi$. 
If $k < r-1$, then $\overline \F$ is reflexive \CD2; if $k=r-1$, then $\overline \F$ is the twisted ideal sheaf of a smooth curve. 

(b) When $\dim X = 3$, our argument for Theorem \ref{filterreflexive} shows that if $\F$ is a rank $r$ reflexive \CD2 sheaf, 
$\E$ a bundle of rank $r-1$ and $\B$ globally generated by a finite dimensional vector space and $\alpha_i \geq 2$ for $1 \leq i < n$, 
then there is an injective map $\phi: \E \to \F$ whose cokernel is the twisted ideal of a smooth curve. 
 
\em}\end{rmks}

\section{Sheaves that are not globally generated}

We give smoothing results for sections of \CD2 reflexive sheaves which are not globally generated. 
When $X = \mathbb P^4$, Theorem \ref{notgg} strengthens Theorem \ref{filterreflexive} to  in a way allow more applications. 
We adopt the following hypothesis.

\begin{hypothesis}\label{setup} 
Let $\G$ be a \CD2 reflexive sheaf of rank $r$ on a smooth fourfold $X$ with singular scheme 
$\sing \G$ having integral curve components and suppose there is an $N$-dimensional 
subspace $V \subseteq H^0 (X,\G)$ for which the cokernel $\Q$ of the evaluation map 
\begin{equation}
V \otimes \O_X \to \G \to \Q \to 0
\end{equation}
is locally principal with $C = \supp \Q$ of dimension $\leq 1$, $C \cap \sing \G = \emptyset$, and 
the curve components of $C$ are integral with 
$\Q$ generically a line bundle along them. 
\end{hypothesis}

The conditions on $\G$ away from $C$ in Hypothesis \ref{setup} are the same as the hypothesis of Theorem \ref{rank2}, 
where we understand $\G$ well. 
The novelty in this section is the analysis of $\G$ near $C$, where $\G$ is locally free. 
The reader may want to take $\G$ a vector bundle and trust that the methods of 
$\S 2$ will work away from $C$. 
 
\begin{lem}\label{two}
With Hypothesis \ref{setup}, let $A \subset X$ be closed with $A \cap (C \cup \sing \G) = \emptyset$ and $\dim A \leq 1$.
If $W \subset V$ is a general subspace of dimension $k < r-1$, then $\phi_W: W \otimes \O_X \to \G$ is injective, 
$\overline \G = \cok \phi_W$ satisfies Hypothesis \ref{setup}, and $\sing \overline \G \cap (C \cup A) = \emptyset$. 
\end{lem}

\begin{proof} Let $K(p) \subset V$ be the kernel of the map $V \to \G (p) = \G_p \otimes k(p)$ for $p \in X - \sing \G$. 
Then $\dim K(p) = N-r$ for $p \not \in C$ and $\dim K(p) = N-r+1$ for $p \in C$. Define 
\[
Z = \{(W, p)\ | W \to \G(p) \text{\ has non-trivial kernel}\} \subset \mathbb G (k,V) \times (X - \sing \G),
\] 
where $(W \subset V) \in \mathbb G (k,V)$. 
For $p \in C$, the fibre $Z_p = \{W \in \mathbb G(k, V) | W \cap K(p) \neq 0\}$ is a Schubert variety of dimension 
$(N-r) + (k-1)(r-k)$, so $Z_p \subset \mathbb G (k,V)$ has codimension 
$N(k-1)+2r-k-rk = (N-r)(k-1)+(r-k) \geq 2$ because of the hypothesis $k < r-1$. 
Therefore $\bigcup_{p \in C} Z_p \subset \mathbb G(k, V)$ is a proper closed set, so the general 
$k$-dimensional subspace $W \in V$ 
yields $\phi_W: W \to \G(p)$ injective for all $p \in C$. 
Therefore $\dim_{k(p)} \cok \phi_W (p) = r-k$ for $p \in C$ and  
$\overline \G$ is a vector bundle on $C$, hence on an open neighborhood of $C$. 
Corollary \ref{avoidance} tells us that $\overline \G$ is a reflexive \CD2 sheaf on $X-C$ with curve components of 
$\sing \overline \G$ integral for general $W$ and that $\sing \overline \G \cap A = \emptyset$. 
Intersecting Zariski open sets of maps in $V$ shows that $\overline \G$ is reflexive \CD2 reflexive with $\sing \overline G$ having integral curve components and locally free in a neighborhood of $C \cup A$ for general $W$.
The map $W \otimes \O_X \to \G$ is injective because its kernel is torsion and contained in $W \otimes \O_X$. 
For the space of sections $\overline V \subset H^0 (\overline \G)$ in Hypothesis \ref{setup}, 
apply the snake lemma to 
\[
\begin{array}{ccccccccc}
0 & \to & W \otimes \O_X & \to & \G & \to & \overline \G & \to & 0 \\
& & || & & \uparrow & & \uparrow & & \\
0 & \to & W \otimes \O_X & \to & V \otimes \O_X & \to & \overline V \otimes \O_X & \to & 0.
\end{array}
\]
\end{proof}

The argument in Lemma \ref{two} uses the assumption $\dim C \leq 1$ from Hypothesis \ref{setup}, not the stronger condition that $\Q$ be a generic line bundle along a smooth curve. The following Lemma uses the full strength of Hypothesis \ref{setup}. 

\begin{lem}\label{one} With hypothesis \ref{setup}, if $\rk \G = 2$ and $\Q$ is a generic line bundle on a smooth curve, 
then a general section $s \in V$ vanishes along a smooth surface. 
\end{lem}

\begin{proof}
First assume $\Q = \L_C$ is a line bundle on a smooth curve $C$. 
Let $U = X - \sing \G$ and define $Z \subset V \times U$ by 
$Z = \{ (s,p)| s(p)=0\}$, where $s(p): V \to \G (p)$ is the map induced by $s \in V$.
Since $V \otimes \O_p \to \G_p$ is surjective for $p \not \in C$, the fiber $Z_p$ of $Z$ over 
$p$ is isomorphic to ${\mathbb A}^{N-2}$, the affine space given by the kernel. 
Therefore $Z \to X$ is a smooth ${\mathbb A}^{N-2}$-bundle of dimension $N+2$ away 
from $\pi_2^{-1} (C)$. 

Now consider $p \in C$, so that $Z_p \subset V$ is a subspace of dimension $N-1$. 
Following Horrocks and Mumford \cite[Proof of Theorem 5.1]{HM}, there is an open affine neighborhood $U^\prime=\Spec R$ of $p$ on which 
$\G$ trivializes as $Re_1 \oplus Re_2$ and $\Q = \mathcal L_C$ as $R/J\ e$, where 
$J = (y_1,y_2,y_3)$ is the ideal of $C$ in $R$. 
After possibly changing free basis for $\G$, we can arrange that the map $\G \to \Q$ is given by 
 $e_2 \mapsto e$ and $e_1 \mapsto 0$ so that the kernel $K$ is $R e_1 \oplus I_C$ generated by 
 $e_1, y_1 e_2, y_2 e_2$ and $y_3 e_2$. 
We can find a basis $v_1, v_2, \dots, v_N$ of $V$ such that $\phi_{U} : V \otimes_k R \to \G_{U^\prime}$ is 
given by $v_1 \mapsto e_1, v_2 \mapsto f_2 e_1+y_1e_2, 
v_3 \mapsto f_3 e_1+y_2 e_2,
v_4 \mapsto f_4 e_1+y_3 e_2$ and $v_i \mapsto f_i e_1+ g_i e_2$ for $5 \leq i \leq N$, 
where $f_i \in \m_p$ and $g_i \in \m_p I_C$. 
Here $\{y_1, y_2, y_3\}$ extend to a regular sequence of parameters for $\O_p$ by appending $y_4$ so that 
$\m_p = (y_1,y_2,y_3,y_4)$. 

Thinking of $V \cong \mathbb A^N$ with coordinate functions $x_i$, a section $s \in V$ can be written 
$s=\sum x_i v_i$ with image in $\G$ over $U^\prime$ being 
\[
x_1 e_1 + x_2 (y_1 e_2 + f_2 e_1) + x_3 (y_2 e_2 + f_3 e_1) + x_4 (y_3 e_2 + f_4 e_1) + \sum_{i=5}^N x_i (g_i e_2 + f_i e_1).
\] 
Therefore the equations for $Z \subset V \times U^\prime$ are  
$$ \begin{aligned} F_1=x_1 +  \sum_{i=2}^N x_i f_i = 0  \\
				F_2=x_2 y_1 + x_3 y_2 + x_4 y_3 + \sum_{i=5}^N x_i g_i = 0.
\end{aligned}$$

The maximal ideal $\m$ of the point $P=(v_2,p) \in V \times \Spec R = \Spec R [x_1, \dots, x_N]$ 
is 
\[
\m = (y_1,y_2,y_3,y_4,x_1,x_2-1,x_3, \dots, x_N)
\]
with system of parameters shown. Clearly $F_1, F_2$ are linearly independent in $\m/\m^2$, hence 
$P$ is a smooth point of $Z$, so that 
$\dim \sing Z \cap Z_p \leq N-2$. Therefore $\dim \sing Z \leq N-1$ and $\pi_1 (\sing Z) \subset V$ 
is proper. 
By generic smoothness, the general fibre of $Z \subset V \times U \to V$ is smooth, which gives the smooth zero locus of  general section in $V$ along $U$. 
Applying Theorem \ref{rank2} (b) to the restriction of $\G$ to $X - C$ shows that the cokernel of 
$W \otimes \O_X \to \G$ has the same property away from $C$. 
Intersecting these Zariski open conditions in $V$ gives the conclusion on all of $X$. 

Now suppose that $\Q$ is a line bundle on a smooth curve except for finitely many points $\{p_1, \dots, p_n\}$.
Since $\Q_{p_i}$ is generated by one element, a general section $s$ doesn't vanish at the $p_i$ and we can apply the argument above 
on the open set $U = X - \{p_1, \dots, p_n\}$. 
\end{proof}

Corollary \ref{P4} extends Theorem \ref{rank2} when $\E$ is a direct sum of line bundles and $X = \mathbb P^4$. 

\begin{cor}\label{P4}
Assume Hypothesis \ref{setup} with $X = \mathbb P^4$ and 
$A \subset X$ closed with $\dim A \leq 1$ and $A \cap (C \cup \sing G) = \emptyset$. 
If $0 \leq a_1 \leq \dots \leq a_m$ and $m < r$, then there is an injective map 
$\phi: \oplus_{i=1}^m \O (-a_i) \to \G$ such that 
\begin{enumerate}
\item[(a)] If $m < r-1$, then $\overline \G = \cok \phi$ is a reflexive \CD2 sheaf of rank $r-m$, which is locally free along $C \cup A$. 
Moreover, there exists $\overline V \subset H^0 (\overline \G (a_m))$ such that $\overline V \otimes \O_X \to \overline \G (a_m) \to \Q (a_m) \to 0$ is exact.
\item[(b)] If $m= r-1$, then $\overline \G = \cok \phi$ is the twisted ideal sheaf of a smooth surface. 
\end{enumerate}
\end{cor}

\begin{proof}
We induct on $m$. For $m=1$, notice that $\G (a_1)$ is a \CD2 reflexive sheaf of rank $r$ with $\sing \G (a_1) = \sing \G$. 
The natural surjection 
$H^0 (\O (a_1)) \otimes \O \to \O (a_1)$ gives 
\begin{equation}\label{diag1}
\begin{array}{cccccccc}
V \otimes H^0 (\O (a_1)) \otimes \O & \to & \G (a_1) & \to & \mathcal Q_1 & \to &  0 \\
\downarrow & & \downarrow &  & \downarrow & &  &  \\
V \otimes \O (a_1) & \to & \G (a_1) & \to & \Q (a_1) & \to  & 0. 
\end{array}
\end{equation}
The vertical arrow on the left is surjective, hence the scheme-theoretic images of the bundles on the left are the same in 
$\G (a_1)$ and the vertical map on the right is an isomorphism, so 
$\Q_1 = \Q (a_1)$ is a generic line bundle on the smooth curve $C$ and we take $V_1 \subset H^0 (\G (a_1))$ to be the image of 
$V \otimes H^0 (\O (a_1))$. 
Lemma \ref{two} or Lemma \ref{one} shows that the general map 
$\O \to \G (a_1)$ has cokernel as stated.

Now suppose $m>1$. Taking $V_1 \subset H^0 (\G (a_1))$ as above, apply 
Lemma \ref{two} to $\G (a_1)$ to see that a general section  $\O \to \G (a_1)$ has a reflexive \CD2 
quotient $ \F (a_1)$ of rank $r-1$ and locally free on $C \cup A$. 
Then $\overline V_1 \subset H^0 ( \F (a_1))$ has the property that
the cokernel of $\overline V_1 \otimes \O \to  \F (a_1)$ is $\Q (a_1)$. 
The induction hypothesis gives an injective map $\phi^\prime: \oplus_{i=2}^m \O (a_1-a_i) \to \F(a_1)$ to obtain a quotient as 
in statements (a) or (b). Then twist and combine with the section $\O \to \G (a_1)$ to obtain a map $\phi$ with 
the properties stated. 
\end{proof}

\subsection{The canonical filtration and filtered Bertini theorem}

Take $\G$ a rank $r$ reflexive \CD2 sheaf on $X = \mathbb P^4$ as in Hypothesis \ref{setup} and $\Q$ a generic line bundle on a smooth curve. Further assume $H^0 (\G (-1))=0$ and  
consider maps $\phi: \E \to \F$ where $\E$ is a direct sum of line bundles and $\F$ is a direct sum of $\G$ with lines bundles. 
We order the summands so that 
\[
\E = \bigoplus_{i=1}^k \O (-a_i) \text{ and } \F = \bigoplus_{j=1, j \neq M}^N \O (-b_j) \oplus \G,
\]
where the $a_i, b_j$ are non-decreasing and $b_j < 0$ for $1 \leq j < M$, $b_j \geq 0$ for $M < j \leq N$. 
Set $b_M = 0$ so that the $b_i$ are non-decreasing and order the summands of $\F$ by 
\[
\K_1 = \O (-b_1), \K_2 = \O (-b_2), \dots \K_M = \G, \K_{M+1} = \O (-b_{M+1}), \dots \K_{N} = \O (-b_{N}).
\]
As in \cite[Example 2.1]{C}, we define a canonical filtration: 
\begin{defn}\label{CF}{\em 
Let $\F_1 = \oplus_{b_j \leq a_1} \K_j = \oplus_{j=1}^{m_1} \K_j$ and set 
 $r_1 = \min\{r: b_{m_1+1} \leq a_{r+1} \}$ and let $\E_1 = \oplus_{i =1}^{r_1} \O (-a_i)$. 
In a similar vein, we next set $\F_2 = \oplus_{b_j \leq a_{r_1+1}} \K_j = \oplus_{j=1}^{m_2} \K_j$, 
$r_2 = \min\{r: b_{m_2+1} \leq a_{r+1}\}$, $\E_2 = \oplus_{i =1}^{r_2} \O (-a_i)$ and continue. 
This gives filtrations $0 = \E_0 \subset \E_1 \subset \dots \subset \E_n = \E$ and 
$0 = \F_0 \subset \F_1 \subset \dots \subset \F_n = \F$. 
Set $\alpha_i = \rk \F_i - \rk \E_i$ for $0 < i < n$ and $\alpha = \rk \F - \rk E$. 
\em}\end{defn}

The subsheaf $\B = \{\phi: \E \to \F: \phi (\E_i) \subset \F_i, 1 \leq i \leq n \} \subset \Hom(\E,\F)$
has a direct sum decomposition $\B \cong \oplus (\C_i^\vee \otimes \F_i) = \oplus \B_i$ as in Remark \ref{filtration}, but if 
$\G$ is a summand of $\F_i$, then $\B_i$ may fail to be globally generated. 

\begin{thm}\label{notgg} 
In the setting above, if $\alpha_i \geq 2$ for $0 < i < n$ and $\alpha = 1$, then the general map 
$\phi: \E \to \F$ is injective and $\cok \phi$ is the twisted ideal sheaf of a smooth surface. 
\end{thm}

\begin{proof} 
We induct on $n$. If $n=1$, 
then Corollary \ref{P4} (b) applies to $\E = \E_1$ and $\F = \F_1$, so we assume $n>1$. 
Letting $t$ be the smallest integer for which $\G$ is a summand of $\F_t$, there are we two cases. 

If $t > 1$, then $\Hom (\E_1, \F_1)$ is globally generated, then Corollary \ref{avoidance} gives an 
injective map $\phi_1: \E_1 \to \F_1$ with \CD2 reflexive cokernel $\overline \F_1$ and locally free along $\cup_{k=2}^n \sing \Q_k \cup C$, 
where $\Q_k = \F_k / \F_{k-1}$.  
Define new filtrations by $\overline \F_k = \F_k / \E_1$ and $\overline \E_k = \E_k / \E_1$. 
The exact sequence $0 \to \overline \F_1 \to \overline \F_2 \to \Q_2 \to 0$ splits, so $\overline \F_2$ is 
reflexive \CD2 by Remark \ref{filtration} and following the split sequences shows this to be true of 
all the $\overline \F_k$ including $\overline \F = \overline \F_n$. 
For this new filtration, $\overline \B_k = \C_k^\vee \otimes \overline \F_k$ is globally generated for $k < t$ and for 
$k \geq t$ there is a space of sections $\overline V$ for $\overline \F_k$ for which the evaluation map 
$\overline V \otimes \O_X \to \overline \F_k$ is $\Q$, because one can take the sum of such sections from $\Q_k$ and 
sections that globally generate $\overline \F_{k-1}$. Thus the induction continues. 

If $t = 1$, then $a_1 \geq 0$ and $\G$ is a summand of $\F_1$. 
We equivalently consider maps $\E (a_1) \to \F (a_1)$ to reduce to the case $a_1 = 0$. 
This is possible because $\G (a_1)$ has a space of sections $V_1$, namely the image of $V \otimes H^0 (a_1)$ in $H^0 (\G (a_1))$, such that 
the cokernel of $V_1 \otimes \O_X \to \G (a_1)$ is $\Q (a_1)$. 
Since $\F_1 (a_1)$ is the direct sum of $\G (a_1)$ and globally generated line bundles, $\F_1 (a_1)$ also has such a space of sections. 
Corollary \ref{P4} gives an injective map 
$\phi_1: \E_1 \to \F_1$ with \CD2 reflexive cokernel $\overline \F_1$ which is locally free along $\cup_{k=2}^n \sing \Q_k \cup C$. 
The filtration $\overline \F_k = \F_k / \E_1$ consists of \CD2 reflexive sheaves and the induction continues. 

The argument in Theorem \ref{filterreflexive} shows that $\phi$ may be taken general since $X = \mathbb P^4$ is projective. 
\end{proof}

\begin{ex}\label{explicitA}{\em
We illustrate the proof with a concrete example. 
Let $\G$ be the Horrocks-Mumford bundle \cite{HM} on $\mathbb P^4$, thus $V = H^0(\G)$ is $4$-dimensional and 
the cokernel of the evaluation map $V \otimes \O_X \to \G$ is a line bundle on a smooth curve $C$ consisting of $25$ lines (see Example \ref{HMeven}). 
The general section of $V$ vanishes along an abelian surface $S \subset \P^4$ of degree ten. 
Consider a general map $\phi: \E \to \F$ where 
\[
\E = \O (3) \oplus \O^2 \oplus \O (-1) \oplus \O (-4)^2 \oplus \O (-5) \text{ and } \F = \O(5) \oplus \O (4)^2 \oplus \O (1)^2 \oplus \G \oplus \O (-3).
\] 
The canonical filtration is given by $\E_1 = \O (3), \E_2 = \E_1 \oplus \O^2 \oplus \O(-1), \E_3 = \E$
and $\F_1 = \O (5) \oplus \O (4)^2, \F_2 = \F_1  \oplus \ \O (1)^2  \oplus \G, \F_3 = \F$. 
Thus $\alpha_1 = 2$ and $\B_1 = \Hom (\E_1, \F_1)$ is globally generated, so by Corollary \ref{avoidance} 
the cokernel $\overline \F_1$ of the general map $\phi_1: \E_1 \to \F_1$ is a \CD2 reflexive sheaf which is locally free along $C$. 
Taking the quotient of the filtrations by $\E_1$ gives new filtrations $\overline \E_2 \subset \overline \E_3$ 
and $\overline \F_2 \subset \overline \F_3$ where the $\overline \E_i$ are direct sums of line bundles and the 
$\overline \F_i$ are \CD2 reflexive sheaves. 
Here $\overline \F_2 = \overline \F_1 \ \oplus \ \O (1)^2  \oplus \G$ has a space $V_1$ of sections for which the cokernel of the evaluation map 
$V_1 \otimes \O_X \to \overline \F_2$ is a line bundle on the smooth curve $C$, namely the sum of global sections generating 
$\overline \F_1$ and $V$. This illustrates the first case in the proof. 

The sheaf $\overline \B_2 = \Hom(\overline \E_2, \overline \F_2)$ is not globally generated, but using the space $V_1$ of sections noted above, Corollary \ref{P4} gives a map $\phi_2: \overline \E_2 \to \overline \F_2$ with \CD2 reflexive cokernel 
$\tilde \F_2$ which is locally free along the union of $C$ an the singular scheme of $\overline \F_3$, hence if we quotient by $\overline \E_2$, 
$\tilde \F_3$ is a \CD2 reflexive sheaf and $\tilde \E_3$ is a direct sums of line bundles. Since $\tilde \F_3$ has a good space of sections, 
the cokernel of a general map $\phi_3: \tilde \E_3 \to \tilde \F_3$ is the twisted ideal sheaf of a smooth surface. 
\em}\end{ex}

\begin{rmk}\label{extension2}{\em
As in Remarks \ref{extensions}, Corollary \ref{P4} and Theorem \ref{notgg} can be modified for curves in $\mathbb P^3$ with easier proofs since $\dim \sing \F = 0$. 
Martin-Deschamps and Perrin have done an exhaustive study of this situation \cite{MDP2}. 
\em}\end{rmk}

\section{Applications to linkage theory}

We apply our results to smoothing members of even linkage classes of codimension two subschemes in $\mathbb P^3$ and $\mathbb P^4$. 
To make the connection to linkage theory transparent, we restrict to Hypothesis \ref{setup2} 
where the condition for smoothing in Theorem \ref{FB} and the necessary condition for integrality in \cite{Nint} coincide. 
Theorem \ref{toy} yields even linkage classes of curves in $\mathbb P^3$ and surfaces in $\mathbb P^4$ in which every integral subscheme is smoothable within its even linkage class. 
In particular, this phenomenon holds for the even linkage classes assoicated to the 
Horrocks-Mumford surface in $\mathbb P^4$ (Examples \ref{HMeven} and \ref{HModd}). 

Recall linkage theory \cite{JM, PS}. 
Codimension two subschemes $X,Y \subset \mathbb P^d$ are \textit{simply linked} if their scheme-theoretic union is a complete intersection. 
They are \textit{evenly linked} if there is a chain $X=X_0, X_1, \dots, X_{2n}=Y$ with $X_i$ simply linked to $X_{i+1}$. 
Clearly even linkage forms an equivalence relation and there is a bijection between even linkage (equivalence) classes $\L$ of locally Cohen-Macaulay subschemes in $\mathbb P^d$ and stable equivalence classes of vector bundles $\N$ on $\mathbb P^d$ satisfying $H^1_* (\N^\vee)=0$ \cite{R2}. 
If a minimal rank element $\N_0$ of the stable equivalence class corresponding to $\L$ via \cite{R2} is zero, 
then $\L$ is the class of ACM codimension two subschemes and we understand which classes contain integral or smooth connected subschemes \cite{Nsecant, steffen}, 
so we henceforth assume $\N_0 \neq 0$. 
Then $\N_0$ is unique up to twist and $\L$ has a minimal element $X_0$ in the sense that each 
$X \in \L$ is obtained from $X_0$ by a sequence of basic double links followed by a cohomology preserving deformation through subschemes in $\L$ \cite{BBM, LR, MDP1, Nlink}, 
where a \textit{basic double link} of $X$ has form $Z = X \cup (H \cap S)$, $S$ a hypersurface containing $X$ and $H$ is a hyperplane meeting $X$ properly. 
Each minimal element $X_0$ has a resolution of the form 
\begin{equation}\label{minimalres}
0 \to \oplus \O(-l)^{p_0(l)} \stackrel{\phi_0}{\to} \N_0 \to \I_{X_0} (a) \to 0
\end{equation}
where $p_0: \mathbb Z \to \mathbb N$, $\sum p_0(l) = \rank \N_0 - 1$ and $a \in \Z$ (there is an 
algorithm to compute $p_0$ and $a$ from $\N_0$ \cite{MDP1,Nlink}). 
Each $X \in \L$ has a resolution of the form 
\begin{equation}\label{Ntype}
0 \to \mathcal P \to \N_0 \oplus \Q \to \I_X (a+h) \to 0
\end{equation}
where $\mathcal P, \Q$ are direct sums of line bundles and $h \geq 0$ is the \textit{height} of $X$. 
Conversely, any codimension two $X \subset \mathbb P^d$ with resolution (\ref{Ntype}) is in $\L$. 
The subset $\L_h \subset \L$ of height $h$ elements is a disjoint union of finitely many irreducible sets $H_i$ determined by 
the values of $h^0 (\I_X (t)), t \in \mathbb Z$ for some $X \in H_i$, so we denote these $H_X$. They are locally closed subsets of the Hilbert scheme consisting 
of subschemes in $\L$ with constant cohomology \cite{BM2, Nint, Nfix} (see \cite[$\S$ VII]{MDP1} for space curves). 
Each $Y \in H_i = H_X$ has the same resolution (\ref{Ntype}) as $X$ modulo adding/subtracting the same line bundle summands to $\mathcal P$ and $\Q$. 

To index the cohomology preserving deformation classes $H_X$ of 
$X \in \L$ with $X_0$ minimal, define $\eta_{X}: \mathbb Z \to \mathbb Z$ 
by $\eta_X (l) = \Delta^n h^0 (\I_X (l)) - \Delta^n h^0 (\I_{X_0} (l-h_X))$, where $h_X$ is the height of $X$ \cite[1.15 (b)]{Nint}. 
The function $\eta_X$ satisfies (a) $\eta_X (l) \geq 0$ for $l \in \mathbb Z$, (b) $\sum \eta_X (l) = h_X$, and 
(c) $\eta$ is connected in degrees $< s_0 ({X_0}) + h_X$, where $s_0 (X)$ is the least degree of a hypersurface contiaing $X$, by
\cite[1.8]{Nint}. 
Setting $\inf \eta_X = \min\{l: \eta_X (l) > 0\}$, 
the connectedness condition says that $\eta_X (l) > 0$ for $\inf \eta_X \leq l < s_0 ({X_0}) + h_X$, so the function 
\[
\theta_X (l) = \left\{\begin{array}{ll} \eta_X (l) - 1 & \inf \eta_X \leq l < s_0 ({X_0}) + h_X \\
\eta_X (l) & \text{otherwise}
\end{array}\right.
\]
is non-negative. 

The usefulness of the function $\theta_X$ comes from the fact that if $X$ is integral, 
then $\theta_X$ is connected about $s_0 (X_0) + h_X$; conversely, if $X_0$ is integral and $\theta_X$ is connected about 
$s_0 (X_0)+h_X$, then $X$ deforms to an integral element in $\L$ \cite{Nint, Nfix}. 
We identify a simplified setting where this connectedness condition on $\theta_X$ for integrality for $X$ lines up with 
the condition $\alpha_i \geq 2$ in Theorem \ref{notgg}. 

\begin{hypothesis}\label{setup2} 
Let $\N_0$ correspond to even linkage class $\L$ on $\mathbb P^d$ and suppose that $\G$ is a quotient sheaf of $\N_0$ by a direct sum of line 
bundles giving an exact sequence 
\begin{equation}\label{g}
0 \to \O^{r-1} \to \G \to \I_{X_0} (a) \to 0
\end{equation}
with $X_0$ minimal in $\L$ and $s_0 (X_0) = a$. 
\end{hypothesis}

\begin{rmk}{\em
Hypothesis \ref{setup2} is a strong condition on an even linkage class. 
In trying to understand the linkage theory of the even linkage class of the Horrocks-Mumford 
surface in $\mathbb P^4$ we observed that these conditions hold and that there are
plenty of other examples where it holds as well. 
Under Hypothesis \ref{setup2} we will see in Proposition \ref{equivalence} a nice correspondence between the 
condition $\alpha_i \geq 2$ from the canonical filtration for $X \in \L$ and the invariant $\theta_X$. 
The condition $s_0 (X_0)=a$ is crucial for this connection to hold. 
In practice the sheaf $\G$ often satisfies Hypothesis \ref{setup} as well, in which case 
$X_0$ links directly to a minimal element $X_0^*$ in the dual linkage class by hypersuraces of degree $a$.
In Examples \ref{curve} and \ref{surface} we will see cases where $\G = \N_0$ and sequence (\ref{g}) is sequence (\ref{minimalres}), but will use the flexibility offered by $\G$ being a proper quotient of $\N_0$ in Examples \ref{HMeven} and \ref{HModd}.
\em}\end{rmk}

A codimension two subscheme $X \subset \mathbb P^d$ having a resolution of the form 
\begin{equation}\label{res1}
\begin{array}{ccccccccc} 0 & \to & \bigoplus_{i=1}^{N+r-1} \O (-a_i) & \stackrel{\phi}{\to} & \bigoplus_{j=1}^N \O (-b_j) \oplus \G & \to & \I_X (a+h) & \to & 0 \\
& & || & & || & & & & \\
& & \E & \stackrel{\phi}{\to} & \F & & & & 
\end{array}
\end{equation}
also has a resolution of the form (\ref{Ntype}) because $\G$ is a quotient of $\N_0$ by a sum of line bundles, hence $X \in \L$. 
Since $\Delta^n h^0 (\O (-c+l))$ as a function of $l$ is a step function equal to $0$ for $l < c$ and 
$1$ for $l \geq c$, one can calculate the functions $\eta_X$ and $\theta_X$ in terms of the $a_i$ and $b_j$ appearing in resolution (\ref{res1}). 
In particular, the $a_i$ and $b_j$ determine the subsheaves $0 = \E_0 \subset \E_1 \dots \subset \E_n = \E$ 
and $0 = \F_0 \subset \F_1 \dots \subset \F_n = \F$ in the canonical filtration of Definition \ref{CF}. 
We have the following 
connection between $\theta_X$ and $\alpha_i = \rank \F_i - \rank \E_i$. 

\begin{prop}\label{equivalence}
$\theta_X$ is connected about $a+h \iff \alpha_i \geq 2$ for each $0<i<n$. 
\end{prop}

\begin{proof}
We relate the shape of the graph of the function $\eta_X$ and the $\alpha_i$ from the canonical filtration. 
From exact sequences (\ref{g}) and (\ref{res1}) and the definition of $\eta_X$, we see the formula
\begin{equation}\label{eta}
\eta_X (l) = \left\{\begin{array}{ll} \#\{j: b_j \leq l-a-h\} - \#\{i: a_i \leq l-a-h\} & l < a+h \\
\#\{j: b_j \leq l-a-h\} - \#\{i: a_i \leq l-a-h\}+(r-1) & l \geq a+h \end{array}\right.
\end{equation}

For simplicity, we examine the translated function $\tilde\eta(l) = \eta_X(l+a+h)$, which follows the same pattern, but with $0$ replacing $a+h$. This lines up better with the twists of $\G$ in the canonical filtration. 

With the notation of Definition \ref{CF}, the summands $\O$ and $\G$ do not appear in $\E_1, \F_1$ if $a_{r_1} < 0$. 
In constructing $\E_1 = \oplus_{i=1}^{r_1} \O (-a_i)$ and $\F_1 = \oplus_{j=1}^{m_1} \O(-b_j)$, since $b_{m_1}\leq a_1$ 
the function $\tilde\eta(l)$ increases by $1$ at $l=b_j$ for each summand $\O (-b_j)$ added to $\F_1$ and decreases by $1$ at $l=a_i$ for each 
summand $\O (-a_i)$ added to $\E_1$. 
Thus $\tilde\eta (l)=0$ for $l \ll 0$, $\tilde\eta$ is non-decreasing up to $l=b_{m_1}$ and then is non-increasing up to $l=a_{r_1}$, where the value is 
$\tilde\eta (a_{r_1}) = \alpha_1$. 
Similarly if $a_{r_2} < 0$, then $\tilde\eta$ increases at the new summands $\O (-b_j)$ added to $\F_2$ and decreases at the summands 
$\O (-a_i)$ added to $\E_2$, so $\tilde\eta$ increases, then decreases to the value 
$\tilde\eta (a_{r_2}) = \alpha_2$ and so on. We conclude that the $\alpha_i$ are the local minimum values of the function 
$\tilde\eta$ in the range $l < 0$.

Let $\E_k$ be the largest summand of $\E$ with terms $\O(-a_i)$ such that $a_i<0$. 
Hence $a_{r_k} <0$, $b_{m_k} <0$ and  $a_{r_k+1}\geq 0$. As before, in the range $[a_{r_k}, a_{r_k+1}-1]$, $\tilde\eta$ is non-decreasing  and is then non-increasing on the interval $[a_{r_k+1}-1, a_{r_{k+1}}]$.  
Since $a_{r_{k+1}} \geq 0$ and the bump in the value of $\tilde \eta$  is just $r-1$ and not the rank of $\G$ which now is a summand of $\F_{k+1}$, we see that $\tilde\eta(a_{r_{k+1}})$ equals $\alpha_{k+1}-1$.

The same is true for all higher $\tilde \eta (a_{r_i}), k<i<n$. In conclusion, we can translate back to $\eta_X$ and say that local minimum values of $\eta_X$ are achieved at $a_{r_i}+a+h, 1\leq i<n$ and $\eta_X(a_{r_i}+a+h) = \alpha_i$ if $a_{r_i}<0$ and $\eta_X(a_{r_i}+a+h) = \alpha_i -1$ if $a_{r_i}\geq 0$.

From the above interpretation, the condition $\alpha_i \geq 2$ is equivalent to the two conditions (a) for $l < a+h$, 
the local minimum values of $\eta_X$ are $\geq 2$, which says that if $\eta_X (l)$ reaches a value of at least two, it remains at least two until $l=a+h-1$ 
and (b) for $l \geq a+h$, the local minimum values of $\eta_X$ are $\geq 1$, which says that if $\eta_X (l)=0$ for some $l \geq a+h$, then it remains zero for larger $l$. 
Condition (a) is equivalent to $\theta_X$ being connected in degrees $< a+h$ and (b) is equivalent to $\theta_X$ being connected in degrees 
$\geq a+h$. Combining, we see that $\alpha_i \geq 2$ for each $i<n$ if and only if $\theta_X$ is connected about $a+h$.  
\end{proof}

\begin{ex}\label{explicitB}{\em The Horrocks-Mumford bundle $\G$ is a quotient of $\N_0$ and there is an exact sequence 
$0 \to \O \to \G \to \I_{X_0} (5) \to 0$, where $X_0$ is the Horrocks-Mumford surface (see Example \ref{HMeven}). 
Example \ref{explicitA} shows that $\alpha_i \geq 2$ for the canonical filtration associated to the vector bundles 
\[
\E = \O (3) \oplus \O^2 \oplus \O (-1) \oplus \O (-4)^2 \oplus \O (-5) \text{ and } \F = \O (5) \oplus \O (4)^2 \oplus \O (1)^2 \oplus \G \oplus \O (-3).
\] 
Theorem \ref{notgg} shows that if $\phi: \E \to \F$ is general, then $\cok \phi$ is $\I_X (5+23)$ for a smooth surface $X$. 
From the resolutions for $X_0$ and $X$, one can read off the function $\eta_X$ and $\theta_X$ is seen to be connected about $28$. 
\em}\end{ex}

For curves in $\mathbb P^3$ or surfaces in $\mathbb P^4$, we combine with Theorem \ref{notgg} to obtain the following smoothing theorem. 

\begin{thm}\label{toy} Assume $d=3$ or $d=4$ and that $\G$ satisfies Hypothesis \ref{setup} and \ref{setup2}. 
Then every integral $X \in \L$ is smoothable in $\L$.
\end{thm}

\begin{proof} 
Let $X_0$ be a minimal element of $\L$. 
If $X \in \L$ is integral, then $\theta_X$ is connected about $a+h$, where $a=s_0(X_0)$ and $h$ is the 
height of $X$ \cite[Theorem 3.4]{Nint}. 
By \cite{BBM, MDP1, Nlink}, there is a sequence of basic double linkages 
starting from $X_0$ to $X_1$ and a  cohomology preserving deformation from $X_1$ to $X$ through 
subschemes in $\L$, so that $X_1 \in H_X$. 
In particular, $\theta_X = \theta_{X_1}$.

Starting from (\ref{g}), we also find a resolution for $X_1$ of the form (\ref{res1}) using $\G$. 
Since $\theta_{X_1}$ is connected about $a+h$,  we can apply Proposition \ref{equivalence} to the 
canonical filtration of $\E, \F$ for this resolution of $X_1$ to find that $\alpha_i \geq 2$ for $i < n$. 
Hence Theorem \ref{notgg} shows that a general deformation of the map $\phi: \E \to \F$ yields 
$X_2$ smooth in $\L$ and the resolution for $X_2$ shows that $X_2 \in H_X$ as well, so that $X$ deforms 
with constant cohomology to $X_2$ through subschemes in $\L$. 
\end{proof}

\begin{rmk}{\em 
If $\G$ satisfies Hypotheses \ref{setup} and \ref{setup2}, then sequence (\ref{g}) assures that $X_0$ may be taken smooth by deforming  
$\phi: \O^{r-1} \to \G$ by Corollary \ref{P4} (or Lemma \ref{two} followed by Lemma \ref{one}), but $X_0$ may fail to be connected. 
This is seen in Example \ref{curve} (a). 
\em}\end{rmk}

\begin{ex}\label{curve}{\em
We give applications to space curves and compare with the literature. 

(a) If $\Omega = \Omega_{\mathbb P^3}$ is the sheaf of differentials on $\mathbb P^3$, then there is a sequence 
\[
0 \to \O^2 \to \Omega (2) \to \I_{X_0} (2) \to 0
\]
where $X_0$ is a pair of skew lines. 
A general quotient of $\G_1 = \Omega (2)$ by a section is a rank two bundle $\G_2$, a twisted null-correlation bundle 
with a sequence $0 \to \O \to \G_2 \to \I_{X_0} \to 0$ as above. 
Both $\G_1$ and $\G_2$ satisfy the hypotheses of Corollary \ref{toy}, so all integral curves 
in the even linkage class of two skew lines are smoothable. 

(b) The vector bundle $\G = \Omega (2)^{\oplus n}$ also satisfies the hypothesis of Corollary \ref{toy} and there is an exact sequence 
$0 \to \O^{3n-1} \to \G \to \I_{X_0} (2n) \to 0$ with $X_0$ a minimal arithmetically Buchsbaum curve, so every integral curve in 
the corresponding even linkage class $L_n$ is smoothable. Bolondi and Migliore classified the smooth curves in $L_n$ of 
maximal rank \cite{BM}. 

(c) More generally, an even linkage class $\L$ of an arithmetically Buchsbaum space curve corresponds to a vector bundle of the form 
$\oplus_{i=1}^q \Omega (a_i)$ \cite{C3}. Chang determines exactly which curves in $\L$ are smoothable \cite{C2}. 
Later Paxia and Ragusa confirmed that all integral curves in these even linkage classes are smoothable \cite{PR}. 

(d) Four general forms $f_i$ of degree $d$ define a rank three bundle $\tilde \Omega$ via 
\[
0 \to \tilde \Omega \to \O (-d)^4 \stackrel{(f_1, f_2, f_3, f_4)}{\to} \O \to 0.
\]
The bundle $\K = \oplus_{i=1}^r \tilde \Omega (2d)$ satisfies the hypotheses of Corollary \ref{toy}, 
hence all integral curves in the corresponding even linkage class are smoothable. This seems to be new. 

(e) If we use forms $f_i$ of different degrees in part (d), the results change. 
For the even linkage class $\L$ corresponding to the rank three bundle $\tilde \Omega$, 
Martin-Deschamps and Perrin determined all smoothable curves in $\L$ \cite{MDP2} and all the curves that deform to integral 
curves is also known \cite[$\S 6$]{Nfix}. It is rather uncommon that these answers agree. 
For example, if $\deg f_1 = \deg f_2 = 1$ and $\deg f_3 = \deg f_4 = 3$, the corresponding even linkage class has integral curves that are not smoothable in $\L$ \cite[$\S 6$]{Nfix}. 
Hartshorne showed that one family of these integral curves forms an irreducible component in the Hilbert scheme whose curves 
cannot even be smoothed in the full Hilbert scheme, much less in $\L$ \cite{H}. 
\em}\end{ex}

\begin{ex}\label{surface}{\em
Much less is known about surfaces in $\mathbb P^4$. 
Let $f_1, \dots, f_5$ be general degree $d$ forms and define $\tilde \Omega$ via 
\[
0 \to \tilde \Omega \to \O (-d)^5 \stackrel{(f_1,f_2,f_3,f_4,f_5)}{\longrightarrow} \O \to 0.
\]
Theorem \ref{toy} applies to the rank four bundle $\G = \tilde \Omega (2d)$ (as well as for any sum $\oplus \Omega (2d)$), 
hence every integral surface in the corresponding even linkage class is smoothable. The case $d=1$ recovers some results of Chang, who more generally determines which surfaces 
in an even linkage class $\L$ corresponding to $\oplus_{i=1}^q \Omega^{p_i} (a_i), p_i \in \{1,2\}$, are smoothable. 
She shows \cite{C2} that any integral arithmetically Buchsbaum surface is smoothable. Her proof can be copied for the case $d>1$, where $\tilde \Omega^{p_i}, p_i \in \{1,2\}$, will be the syzygy bundles in the resolution.
It is not known what conditions on $d_i = \deg f_i$ ensure that integral surfaces are smoothable when the $d_i$ are not constant, see Question \ref{Q2}. 
\em}\end{ex}

\begin{ex}\label{HMeven}{\em
Let $\F_{HM}$ be the much studied Horrocks-Mumford bundle on $\mathbb P^4$ \cite{HM}. It is known that $\F_{HM}$ has a $4$-dimensional space of sections $V = H^0 (\F_{HM})$ and that the general section $s \in V$ defines an abelian surface $X_{HM}$  of degree ten, the Horrocks-Mumford surface, via an exact sequence 
\begin{equation}\label{HMsurface}
0 \to \O \stackrel{s}{\to} \F_{HM} \to \I_{X_{HM}} (5) \to 0.
\end{equation}
The normalization $\F_{HM} (-3)$ is the homology of a self-dual monad 
\begin{equation}\label{HMmonad}
0 \to \O (-1)^5 \stackrel{A^\vee}{\to} \bigoplus_{i=1}^2 \Omega^2 (2) \stackrel{A}{\to} \O^5 \to 0 
\end{equation}
where $\Omega = \Omega^1_{\mathbb P^4}$ and $\Omega^2 = \wedge^2 \Omega$.

Let $K$ be the kernel of the map $\bigoplus_{i=1}^2 \Omega^2 (2) \stackrel{A}{\to} \O^5 \to 0$.  
Then $K$ is a vector bundle of rank $7$ and $H^3_*(K)=0$. Furthermore, $K$ has no line bundle summands, 
because a line bundle summand $\O(a)$ of $K$ induces an $\O(a)$ summand for $\bigoplus_{i=1}^2 \Omega^2 (2)$, 
but $\Omega^2$ is indecomposable, \cite[pp. 86--88]{OSS}. 
Hence $K(3)$ is the minimal (up to twist) bundle $\N_0$ for an even linkage class $\L$ of surfaces. 
Manolache \cite{M}  calculates the minimal generators of $H^0_*(K)$ and a minimal resolution 
\begin{equation}\label{Manolache}
 0 \to B \to  \O(2)^5 \oplus \O^4 \oplus \O (-1)^{15}  \to K(3) \to 0.
 \end{equation}
Here $H^1_*(B)=0$ and $B$ has no line bundle summands, so $B^\vee$ is the minimal bundle $\N_0^*$ for the odd linkage class $\L^*$ corresponding to $\L$.

Since $0 \to \O(2)^5 \to K(3) \to \F_{HM} \to 0$ is exact, we see that $X_{HM}$ has the resolution
\begin{equation}\label{Kres}
0 \to \O (2)^5 \oplus \O \to K (3) \to \I_{X_{HM}} (5) \to 0.
\end{equation} 

This shows that $\L$ is the even linkage class of $X_{HM}$ and that $X_{HM}=X_0$ is a minimal surface in $\L$. It also shows that $\F_{HM}$ successfully plays the role of $\G$ in Hypothesis \ref{setup2} with $r=4, a=5$.

Horrocks and Mumford show  that the evaluation map $V \otimes \O_{\mathbb P^4} \stackrel{ev}{\to}  \F_{HM}$ is surjective away from 
a smooth curve $C$ consisting of $25$ disjoint lines \cite[Theorem 5.1]{HM}. 
For a point $x \in C$, they find a local basis $e_1, e_2$ for $\F_{HM}$ and a basis $s, t, t^\prime, t^{\prime \prime}$ for the vector space $V$ 
such that the local matrix for the map $ev$  is $\left(\begin{array}{cccc} 1 & f & f^\prime & f^{\prime\prime} \\
0 & u & u^\prime & u^{\prime\prime} \end{array}\right)$, where $(u,u^\prime, u^{\prime\prime})$ generate the local ideal on $C$,  showing that  $\cok ev$ is locally $\O_C$, hence the cokernel of $V \otimes \O \stackrel {ev}{\to} \F_{HM}$ is a line bundle $L_C$ on the smooth curve $C$. 

Thus $(\F_{HM}, V)$ fit the requirements of Hypotheses \ref{setup2} and \ref{setup}. By Theorem \ref{toy},  every integral surface in $\L$ is smoothable. 
\em}
\end{ex}

\begin{ex}\label{HModd}{\em
Now we treat the dual class $\L^*$ for the Horrocks-Mumford surface. We have the long exact sequence
$$ 0 \to \ker ev \to V \otimes \O \to \F_{HM} \to L_C \to 0.$$

$\ker ev$ is a rank two reflexive sheaf, locally free away from $C$.  
Since $\Ext^i(L_C, \O) =0$ for $i = 0,1,2$, dualizing gives the sequence (with $\G = (\ker ev)^\vee$),
\begin{equation}\label{dualclass}
0 \to \F_{HM}^\vee \stackrel {ev^\vee}{\to} \O^4 \to \G \to 0.
\end{equation}

The Fitting ideal of the local matrix for ${ev^\vee}$ at a point $x\in C$ shows that $\sing \G $ is the scheme $C$. Hence $\G$ is a $\CD2$ reflexive rank two sheaf, generated by its global sections. 

From sequence (\ref{Manolache}), we also obtain the exact sequence
$$ 0 \to \B \to \O^4 \oplus \O (-1)^{15} \to \F_{HM} \to 0.$$
Comparing the dual of this sequence with sequence (\ref{dualclass}), we get the exact sequence
$$ 0 \to \O(1)^{15} \to \B^\vee \to \G \to 0.$$
This shows that any non-zero section of $\G$ yields a minimal surface $Y_0$ for the dual linkage class. Hence $\G$ satisfies the requirements of Hypotheses  \ref{setup2}, \ref{setup} and Theorem \ref{toy} applies to $\G$, showing 
that every integral surface in $\L^*$ is smoothable.
\em}
\end{ex}

\end{document}